\documentclass[11pt]{article}
\usepackage{amsmath}
\usepackage{amssymb}
\usepackage{amsthm}
\usepackage[usenames]{color}
\usepackage{amscd}
\usepackage{dsfont}
\usepackage{indentfirst}

\usepackage[colorlinks=true,linkcolor=blue,filecolor=red,
citecolor=webgreen]{hyperref}
\definecolor{webgreen}{rgb}{0,.5,0}

\numberwithin{equation}{section}

\def\C{{\mathds{C}}}

\def\N{{\mathds{N}}}
\def\Z{{\mathds{Z}}}

\def\1{{\bf 1}}

\newcommand{\DOT}{\text{\rm\Huge{.}}}

\newtheorem{theorem}{Theorem}[section]

\newtheorem{lemma}[theorem]{Lemma}
\newtheorem{cor}[theorem]{Corollary}

\begin{document}

\title{{\bf Menon-type identities concerning Dirichlet characters}}
\author{L\'aszl\'o T\'oth \\ \\ Department of Mathematics, University of P\'ecs \\
Ifj\'us\'ag \'utja 6, 7624 P\'ecs, Hungary \\ E-mail: {\tt ltoth@gamma.ttk.pte.hu}}
\date{}
\maketitle

\centerline{Int. J. Number Theory {\bf 14}, No. 4 (2018), 1047-1054}

\begin{abstract} Let $\chi$ be a Dirichlet character (mod $n$) with conductor $d$.
In a quite recent paper Zhao and Cao deduced the identity
$\sum_{k=1}^n (k-1,n) \chi(k)= \varphi(n)\tau(n/d)$, which reduces
to Menon's identity if $\chi$ is the principal character (mod $n$).
We generalize the above identity by considering even functions (mod
$n$), and offer an alternative approach to proof. We also obtain
certain related formulas concerning Ramanujan sums.
\end{abstract}

{\sl 2010 Mathematics Subject Classification}: 11A07, 11A25

{\sl Key Words and Phrases}: Menon's identity, Dirichlet character, primitive character, arithmetic function, even function (mod $n$),
Euler's totient function, Ramanujan sum, congruence

\section{Introduction}

In a quite recent paper Zhao and Cao \cite{ZhaCao} derived the following identity. Let $\chi$ be a Dirichlet character (mod $n$)
with conductor $d$ ($d\mid n$). Then
\begin{equation} \label{Menon_id_char}
\sum_{k=1}^n (k-1,n) \chi(k)= \varphi(n)\tau(n/d) \quad (n\in \N),
\end{equation}
where $(k-1,n)$ stands for the greatest common divisor of $k-1$ and $n$, $\varphi(n)$ is Euler's totient function and $\tau(n)=\sum_{d\mid n} 1$ is
the divisor function. If $\chi$ is the principal character (mod $n$), that is $d=1$, then \eqref{Menon_id_char} reduces to
Menon's identity
\begin{equation} \label{Menon_id}
\sum_{\substack{k=1\\ (k,n)=1}}^n (k-1,n) = \varphi(n)\tau(n) \quad (n\in \N).
\end{equation}

On the other hand, if $\chi$ is a primitive Dirichlet character (mod $n$), then \eqref{Menon_id_char} gives (the case $d=n$)
\begin{equation} \label{Menon_id_primit_char}
\sum_{k=1}^n (k-1,n) \chi(k)= \varphi(n) \quad (n\in \N).
\end{equation}

In fact, Zhao and Cao \cite{ZhaCao} first proved formula \eqref{Menon_id_primit_char} and then deduced identity \eqref{Menon_id_char} by
using the fact that every Dirichlet character is induced by a primitive character (Lemma \ref{Lemma_char}). They showed
that the left hand sides of \eqref{Menon_id_char} and \eqref{Menon_id_primit_char} are multiplicative in $n$ and computed their values
for prime powers.

It is the goal of the present paper to generalize these identities
by considering even functions (mod $n$), and to offer an alternative
approach to proof, based on direct manipulations of the
corresponding sums, valid for any integer $n\in \N$.

A function $f:\Z \to \C$ is called an even function (mod $n$) if
$f((k,n))=f(k)$ holds for any $k\in \Z$, where $n\in \N$ is fixed.
The term $n$-even function is also used in the literature. Examples
of even functions (mod $n$) are $f(k)=(k,n)$, more generally
$f(k)=F((k,n))$, where $F$ is an arbitrary arithmetic function;
$f(k)=c_n(k)$, representing the Ramanujan sum; the function $N(k)$,
counting the solutions $(x_1,\ldots,x_q)\in \Z_n^q$ of the
congruence $x_1+\cdots +x_q\equiv k$ (mod $n$) such that
$(x_1,n)=\cdots =(x_q,n)=1$, with $q\in \N$ fixed. General accounts
of even functions (mod $n$) can be found, e.g., in the books by
McCarthy \cite{McC1986}, Schwarz and Spilker \cite{SchSpi1994}, and
the paper by the author and Haukkanen \cite{TotHauk2011}.

Different Menon-type identities were established by several authors. See, e.g., the papers by Haukkanen \cite{Hauk2005}, Li and Kim,
\cite{LiKim2017, LiKim}, Miguel \cite{Mig2016}, Sita Ramaiah \cite{Sit1978}, T\u{a}rn\u{a}uceanu \cite{Tar2012}, the author
\cite{Tot2011, Tot2013}.

\section{Main results}

We prove the following results. The first one is a direct generalization of Menon's identity, not involving characters, which will be
used later in the proof. Let $\mu$ denote, as usual, the M\"obius function and let $*$ denote the Dirichlet convolution of arithmetic functions.

\begin{theorem} \label{Theorem_Menon_direct} Let $n,d\in \N$, $r,s\in \Z$ such that $d\mid n$. Let $f$ be an even function (mod $n$). Then
\begin{equation} \label{Th_gen_Menon}
\sum_{\substack{k=1\\ (k,n)=1\\ k\equiv r \text{\rm (mod $d$)} }}^n f(k-s) =
\begin{cases} \displaystyle \frac{\varphi(n)}{\varphi(d)} \sum_{\substack{e\mid n\\ (e,s)=1 \\ (e,d)\mid r-s}}
\frac{(\mu*f)(e)}{\varphi(e)} \varphi((e,d)),
& \text{ if $(r,d)=1$}, \\ 0, & \text{ if $(r,d)>1$}.
\end{cases}
\end{equation}
\end{theorem}

If $f(k)=F((k,n)$, where $F$ is an arbitrary arithmetic function, $d=1$ and $(s,n)=1$, then from \eqref{Th_gen_Menon} we reobtain the identity
due to Sita Ramaiah \cite[Th.\ 9.1]{Sit1978} in the more general setting of regular arithmetic convolutions.

\begin{cor} \label{Cor_gen_Menon}  Let $n,d\in \N$, $r,s\in \Z$ such that $d\mid n$. Then
\begin{equation} \label{form_gen_Menon}
\sum_{\substack{k=1\\ (k,n)=1\\ k\equiv r \text{\rm (mod $d$)} }}^n (k-s,n) =
\begin{cases} \displaystyle \frac{\varphi(n)}{\varphi(d)} \sum_{\substack{e\mid n \\ (e,s)=1 \\ (e,d)\mid r-s}} \varphi((e,d)),
& \text{ if $(r,d)=1$}, \\ 0, & \text{ if $(r,d)>1$}.
\end{cases}
\end{equation}
\end{cor}

If $d=1$ and $s=1$, then \eqref{form_gen_Menon} reduces to Menon's identity \eqref{Menon_id}.

\begin{cor} \label{Cor_gen_Ramanujan}  Let $n,d\in \N$, $r,s\in \Z$ such that $d\mid n$. Then
\begin{equation} \label{form_gen_Menon_Ramanujan}
\sum_{\substack{k=1\\ (k,n)=1\\ k\equiv r \text{\rm (mod $d$)} }}^n c_n(k-s) =
\begin{cases} \displaystyle \frac{\varphi(n)}{\varphi(d)} \sum_{\substack{e\mid n \\ (e,s)=1 \\ (e,d)\mid r-s}}
\frac{e\mu(n/e)}{\varphi(e)} \varphi((e,d)),
& \text{ if $(r,d)=1$}, \\ 0, & \text{ if $(r,d)>1$}.
\end{cases}
\end{equation}
\end{cor}

If $d=1$, then \eqref{form_gen_Menon_Ramanujan} gives the first identity of the known formulas
\begin{equation*}
\sum_{\substack{k=1\\ (k,n)=1}}^n c_n(k-s) = \varphi(n) \sum_{\substack{e\mid n \\ (e,s)=1}}
\frac{e\mu(n/e)}{\varphi(e)} = \mu(n)c_n(s),
\end{equation*}
the second one being the Brauer-Rademacher identity. See \cite[Ch.\ 2]{McC1986}.

\begin{theorem} \label{Theorem_main_char}
Let $\chi$ be a Dirichlet character {\rm (}mod $n${\rm )} with conductor $d$ {\rm (}$n,d\in \N$, $d\mid n${\rm )}. Let $f$ be an
even function (mod $n$) and  let $s\in \Z$. Then
\begin{equation*}
\sum_{k=1}^n f(k-s) \chi(k) = \varphi(n) \chi^*(s)  \sum_{\substack{\delta \mid n/d\\ (\delta,s)=1}}
\frac{(\mu*f)(\delta d)}{\varphi(\delta d)},
\end{equation*}
where $\chi^*$ is the primitive character (mod $d$) that induces $\chi$.
\end{theorem}

\begin{cor} \label{Cor_main_char_Menon}
Let $\chi$ be a Dirichlet character (mod $n$) with conductor $d$ ($n,d\in \N$, $d\mid n$) and let $s\in \Z$.
Then
\begin{equation} \label{id_Menon_char_s}
\sum_{k=1}^n (k-s,n) \chi(k) = \varphi(n) \chi^*(s)  \sum_{\substack{\delta \mid n/d\\ (\delta,s)=1}} 1.
\end{equation}
\end{cor}

If $s=1$, then \eqref{id_Menon_char_s} reduces to the identity \eqref{Menon_id_char} of Zhao and Cao \cite{ZhaCao}.
If $\chi$ is the principal character (mod $n$), that is $d=1$, then \eqref{id_Menon_char_s} gives
\begin{equation} \label{id_Menon_s}
\sum_{\substack{k=1\\ (k,n)=1}}^n (k-s,n)  = \varphi(n) \sum_{\substack{\delta \mid n\\ (\delta,s)=1}} 1,
\end{equation}
valid for any $s\in \Z$. If $(s,n)=1$, then the right hand side of \eqref{id_Menon_s} is $\varphi(n)\tau(n)$, like in
Menon's classical identity \eqref{Menon_id}.

\begin{cor} \label{Cor_main_char_Ramanujan}
Let $\chi$ be a Dirichlet character (mod $n$) with conductor $d$ ($n,d\in \N$, $d\mid n$) and let $s\in \Z$.
Then
\begin{equation*}
\sum_{k=1}^n c_n(k-s) \chi(k) = d \varphi(n) \chi^*(s)  \sum_{\substack{\delta \mid n/d\\ (\delta,s)=1}}
\frac{\delta \mu(n/(\delta d))}{\varphi(\delta d)}.
\end{equation*}
\end{cor}

We remark that the sums in Theorem \ref{Theorem_main_char} and Corollaries \ref{Cor_main_char_Menon} and \ref{Cor_main_char_Ramanujan}
vanish provided that $(s,d)>1$.

\begin{theorem} \label{Th_prim_char} Let $\chi$ be a primitive Dirichlet character (mod $n$), where $n\in \N$. Let $f$ be an
even function (mod $n$) and  let $s\in \Z$. Then
\begin{equation*}
\sum_{k=1}^n f(k-s) \chi(k) = (\mu*f)(n) \chi(s).
\end{equation*}
\end{theorem}

The above results can be applied to other special even functions (mod $n$), as well. For example, we have

\begin{cor} \label{Cor_prim_char_special} Let $\chi$ be a primitive Dirichlet character (mod $n$), where $n\in \N$. Let $F$ be an arbitrary
arithmetic function and let $s\in \Z$. Then
\begin{equation} \label{sum_F}
\sum_{k=1}^n F((k-s,n)) \chi(k) = (\mu*F)(n) \chi(s).
\end{equation}

In particular,
\begin{equation*}
\sum_{k=1}^n \sigma((k-s,n)) \chi(k) = n \chi(s),
\end{equation*}
where $\sigma(n)$ is the sum-of-divisors function, and
\begin{equation*}
\sum_{k=1}^n \tau((k-s,n)) \chi(k) = \chi(s).
\end{equation*}
\end{cor}

It turns out that if $F$ is a multiplicative function and $s=1$, then the sum \eqref{sum_F}
is also multiplicative in $n$.

The sums in Theorem \ref{Th_prim_char} and Corollary \ref{Cor_prim_char_special} vanish provided that $(s,n)>1$.

\section{Proofs}

We need the following known results. For the first one see, e.g., \cite[Th.\ 9.2]{MonVau2007}.

\begin{lemma} \label{Lemma_char}
Let $\chi$ be a  Dirichlet character (mod $n$) with conductor $d$. Then there is a unique primitive character
$\chi^*$ (mod $d$) that induces $\chi$. That is,
\begin{equation*}
\chi(k) =  \begin{cases} \chi^*(k), & \text{ if $(k,n)=1$}, \\ 0, & \text{ if $(k,n)>1$}.
\end{cases}
\end{equation*}
\end{lemma}

For the next result see, e.g., \cite[Th.\ 9.4]{MonVau2007}. However, it is not included in most of other textbooks.
For the sake of completeness we present its (short) proof.

\begin{lemma} \label{Lemma_char_sum}
Let $\chi$ be a primitive character (mod $n$). Then for any $d\mid n$, $d<n$ and any $s\in \Z$,
\begin{equation*}
\sum_{\substack{k=1\\ k\equiv s \text{\rm (mod $d$)} }}^n \chi(k)=0.
\end{equation*}
\end{lemma}

\begin{proof}[Proof of Lemma {\rm \ref{Lemma_char_sum}}] Since $\chi$ is a primitive character, for a given
$d\mid n$, $d<n$ there exists $c\in \Z$ such that $(c,n)=1$, $c\equiv 1$ (mod $d$) and $\chi(c)\ne 1$. We have
\begin{equation*}
S:=\sum_{\substack{k=1\\ k\equiv s \text{\rm (mod $d$)} }}^n \chi(k)= \sum_{t \text{ (mod $n/d$)}} \chi(s+td).
\end{equation*}

Here, since $(c,n)=1$, as $t$ runs through a complete residue system (mod $n/d$), the numbers $j=cs+tcd$ run through a
complete residue system (mod $n$), where $j\equiv cs\equiv s$ (mod $d$). Hence,
\begin{equation*}
S = \sum_{t \text{ (mod $n/d$)}} \chi(cs+tcd)= \chi(c)\sum_{t \text{ (mod $n/d$)}} \chi(s+td)=\chi(c)S.
\end{equation*}

Since $\chi(c)\ne 1$, it follows that $S=0$.
\end{proof}

\begin{proof}[Proof of Theorem {\rm \ref{Theorem_Menon_direct}}] If $(k,n)=1$ and $k\equiv r$ (mod $d$), then $(r,d)=(k,d)=1$. Therefore,
the sum is empty in the case $(r,d)>1$.

Now assume that $(r,d)=1$. Since $f$ is an even function (mod $n$),
\begin{equation} \label{repr_even_func}
f(k) = f((k,n)) = \sum_{d\mid (k,n)} (\mu*f)(d).
\end{equation}

We have
\begin{equation*}
T:= \sum_{\substack{k=1\\ (k,n)=1\\ k\equiv r \text{\rm (mod $d$)}
}}^n f(k-s) = \sum_{\substack{k=1\\ k\equiv r \text{\rm (mod $d$)}
}}^n f(k-s) \sum_{\delta \mid (k,n)} \mu(\delta)
\end{equation*}
\begin{equation*}
= \sum_{\delta \mid n} \mu(\delta) \sum_{\substack{k=1\\ \delta \mid k \\ k\equiv r \text{\rm (mod $d$)} }}^n f(k-s)
= \sum_{\delta \mid n} \mu(\delta) \sum_{\substack{j=1 \\ \delta j\equiv r \text{\rm (mod $d$)} }}^{n/\delta} f(\delta j-s).
\end{equation*}

According to \eqref{repr_even_func},
\begin{equation*}
T= \sum_{\delta \mid n} \mu(\delta) \sum_{\substack{j=1 \\ \delta j\equiv r \text{\rm (mod $d$)} }}^{n/\delta} \sum_{e\mid (\delta j-s,n)}
(\mu*f)(e)
\end{equation*}
\begin{equation} \label{last}
= \sum_{\delta \mid n} \mu(\delta) \sum_{e\mid n} (\mu*f)(e) \sum_{\substack{j=1 \\ \delta j\equiv r \text{\rm (mod $d$)}\\
\delta j\equiv s \text{\rm (mod $e$)} }}^{n/\delta} 1.
\end{equation}

Let $\delta,d,e$ be fixed. The linear congruence $\delta j\equiv r$ (mod $d$) has solutions in $j$ if and only if $(\delta,d)\mid r$, equivalent to
$(\delta,d)=1$, since $(r,d)=1$. Similarly, the congruence $\delta j\equiv s$ (mod $e$) has solutions in $j$ if and only if $(\delta,e)\mid s$.
The above two congruences have common solutions in $j$ if and only if $(d,e)\mid r-s$. Furthermore, if $j_1$ and $j_2$ are solutions of these
simultaneous congruences, then $\delta j_1\equiv \delta j_2$ (mod $d$) and $\delta j_1\equiv \delta j_2$ (mod $e$). This gives $j_1\equiv j_2$ (mod $d$),
since $(\delta,d)=1$, and $j_1\equiv j_2$  (mod $e/(\delta,e)$). That is, $j_1\equiv j_2$ (mod $[d,e/(\delta,e)]$), the least common multiple of $d$
and $e/(\delta,e)$. We conclude that there are
\begin{equation*}
N=\frac{n}{\delta[d,e/(\delta,e)]}= \frac{n}{[\delta d,[e,\delta]]}= \frac{n}{[\delta d,e]}
\end{equation*}
solutions (mod $n/\delta$). Therefore, the value of the last sum in \eqref{last} is $N$.

We deduce that
\begin{equation*}
T = \sum_{\substack{\delta \mid n\\ (\delta,d)=1}} \mu(\delta) \sum_{\substack{e\mid n\\ (e,\delta)\mid s\\ (e,d)\mid r-s}}
(\mu*f)(e) \frac{n}{[\delta d,e]}
\end{equation*}
\begin{equation*}
=\frac{n}{d} \sum_{\substack{e\mid n\\ (e,d)\mid r-s}}  \frac{(\mu*f)(e)}{e} (d,e) \sum_{\substack{\delta \mid n\\ (\delta,d)=1\\ (\delta,e)\mid s}}
\frac{\mu(\delta)(\delta,e)}{\delta}.
\end{equation*}

Here the inner sum is
\begin{equation*}
\prod_{\substack{p \mid n\\ p\nmid d\\ (p,e)\mid s }} \left(1-\frac{(p,e)}{p} \right),
\end{equation*}
which equals $(\varphi(n)/n)\left(\varphi(de)/de\right)^{-1}$ in the case $(e,s)=1$ and zero otherwise. We obtain that
\begin{equation*}
T = \varphi(n) \sum_{\substack{e\mid n\\ (e,d)\mid r-s\\ (e,s)=1}}  \frac{(\mu*f)(e)}{\varphi(de)} (d,e)
= \frac{\varphi(n)}{\varphi(d)} \sum_{\substack{e\mid n\\ (e,d)\mid r-s\\ (e,s)=1}}  \frac{(\mu*f)(e)}{\varphi(e)} \varphi((d,e)).
\end{equation*}
\end{proof}

\begin{proof}[Proof of Corollary {\rm \ref{Cor_gen_Menon}}] Apply Theorem \ref{Theorem_Menon_direct}. Let $f(k)=(k,n)$.
Then for every $e\mid n$ we have
\begin{equation*}
(\mu*f)(e)=\sum_{ab=e} \mu(a)(b,n)= \sum_{ab=e} \mu(a)b= \varphi(e).
\end{equation*}
\end{proof}

\begin{proof}[Proof of Corollary {\rm \ref{Cor_gen_Ramanujan}}] Apply Theorem \ref{Theorem_Menon_direct}. Select $f(k)=c_n(k)$ and
use the familiar formula
\begin{equation*}
c_n(k)=\sum_{e\mid (k,n)} e\mu(n/e).
\end{equation*}

It follows that $(\mu*c_n(\DOT))(e)= e\mu(n/e)$ for any $e\mid n$. Also see \cite[Sect.\ 3]{TotHauk2011}.
\end{proof}

\begin{proof}[Proof of Theorem {\rm \ref{Theorem_main_char}}] We have, according to Lemma \ref{Lemma_char},
\begin{equation*}
S_f:= \sum_{k=1}^n f(k-s) \chi(k) = \sum_{\substack{k=1\\ (k,n)=1}}^n f(k-s) \chi^*(k)
\end{equation*}
\begin{equation*}
= \sum_{r=1}^d \sum_{\substack{k=1\\ (k,n)=1\\ k\equiv r \text{\rm (mod $d$)} }}^n f(k-s) \chi^*(k)
= \sum_{r=1}^d \chi^*(r) \sum_{\substack{k=1\\ (k,n)=1\\ k\equiv r \text{\rm (mod $d$)} }}^n f(k-s).
\end{equation*}

Now, by using Theorem \ref{Theorem_Menon_direct},
\begin{equation*}
S_f= \frac{\varphi(n)}{\varphi(d)} \sum_{r=1}^d \chi^*(r) \sum_{\substack{e\mid n\\ (e,s)=1\\ (e,d)\mid r-s}} \frac{(\mu*f)(e)}{\varphi(e)}
\varphi((e,d))
\end{equation*}
\begin{equation*}
= \frac{\varphi(n)}{\varphi(d)} \sum_{\substack{e\mid n\\ (e,s)=1}} \frac{(\mu*f)(e)}{\varphi(e)} \varphi((e,d)) \sum_{\substack{r=1\\
r\equiv s \text{ (mod $(e,d)$)}}}^d \chi^*(r).
\end{equation*}

Here, by Lemma \ref{Lemma_char_sum} the inner sum is zero, unless $(e,d)=d$, that is $d\mid e$, and in this case the inner sum is
$\chi^*(s)$. We deduce that
\begin{equation*}
S_f= \frac{\varphi(n)}{\varphi(d)} \chi^*(s) \sum_{\substack{e\mid
n\\ d\mid e\\(e,s)=1}} \frac{(\mu*f)(e)}{\varphi(e)} \varphi(d) =
\varphi(n) \chi^*(s) \sum_{\substack{\delta \mid n/d\\(\delta,s)=1}}
\frac{(\mu*f)(\delta d)}{\varphi(\delta d)},
\end{equation*}
which vanishes if $(s,d)>1$.
\end{proof}

\begin{proof}[Proof of Corollary {\rm \ref{Cor_main_char_Menon}}] Apply Theorem \ref{Theorem_main_char} to $f(k)=(k,n)$, where
$(\mu*f)(e)=\varphi(e)$ for every $e\mid n$.
\end{proof}

\begin{proof}[Proof of Corollary {\rm \ref{Cor_main_char_Ramanujan}}] Apply Theorem \ref{Theorem_main_char} by selecting $f(k)=c_k(n)$.
See the proof of Corollary \ref{Cor_gen_Ramanujan}.
\end{proof}

\begin{proof}[Proof of Theorem {\rm \ref{Th_prim_char}}] This is a direct consequence of Theorem \ref{Theorem_main_char}
by taking $d=n$. A short direct proof is the following: by using \eqref{repr_even_func} and Lemma \ref{Lemma_char_sum} we have
\begin{equation*}
\sum_{k=1}^n f(k-s)\chi(k)= \sum_{k=1}^n \chi(k) \sum_{e\mid (k-s,n)} (\mu*f)(e)
\end{equation*}
\begin{equation*}
=\sum_{e\mid n} (\mu*f)(e) \sum_{\substack{k=1\\ k\equiv s \text{\rm (mod $e$)}}}^n  \chi(k) =  (\mu*f)(n)\chi(s).
\end{equation*}
\end{proof}

\begin{proof}[Proof of Corollary {\rm \ref{Cor_prim_char_special}}] Use Theorem \ref{Th_prim_char}. Select $f(k)=F((k,n))$ and then $F=\sigma$ and
$F=\tau$, respectively.
\end{proof}


\begin{thebibliography}{99}

\bibitem{Hauk2005} P.~Haukkanen, Menon's identity with respect to a generalized divisibility relation, {\it Aequationes Math.} {\bf 70} (2005),
240--246.

\bibitem{LiKim2017} Y.~Li and D.~Kim, A Menon-type identity with many tuples of group of units in residually finite Dedekind domains,
{\it J. Number Theory} {\bf 175} (2017), 42--50.

\bibitem{LiKim} Y.~Li and D.~Kim, Menon-type identities derived from actions of
subgroups of general linear groups, {\it J. Number Theory} {\bf 179} (2017), 97--112.

\bibitem{McC1986} P.~J.~McCarthy, Introduction to Arithmetical Functions, Springer, 1986.

\bibitem{Mig2016} C.~Miguel, A Menon-type identity in residually finite Dedekind domains, {\it J. Number Theory} {\bf 164} (2016), 43--51.

\bibitem{MonVau2007} H.~L.~Montgomery and R.~C.~Vaughan, {\it Multiplicative Number
Theory I. Classical Theory}, Cambridge University Press, 2007.

\bibitem{SchSpi1994} W.~Schwarz and J.~Spilker, {\it Arithmetical
functions, An introduction to elementary and analytic properties of
arithmetic functions and to some of their almost-periodic
properties}, London Mathematical Society Lecture Note Series, 184.
Cambridge University Press, Cambridge, 1994.

\bibitem{Sit1978} V.~Sita Ramaiah, Arithmetical sums in regular convolutions, {\it J. Reine Angew. Math.} {\bf 303/304} (1978), 265--283.

\bibitem{Tar2012} M.~T\u{a}rn\u{a}uceanu, A generalization of Menon's identity, {\it J. Number Theory} {\bf 132} (2012), 2568--2573.

\bibitem{Tot2011} L.~T\'oth, Menon's identity and arithmetical sums representing functions of several variables,
{\it Rend. Sem. Mat. Univ. Politec. Torino} {\bf 69} (2011), 97--110.

\bibitem{Tot2013} L.~T\'oth, Another generalization of the gcd-sum function, {\it Arab. J. Math.} {\bf 2} (2013), 313--320.

\bibitem{TotHauk2011} L. T\'oth  and P. Haukkanen, The discrete Fourier transform of $r$-even functions, {\it Acta Univ. Sapientiae, Math.}
{\bf 3} (2011), 5--25.

\bibitem{ZhaCao} X.-P.~Zhao and Z.-F.~Cao, Another generalization of Menon's identity, {\it Int. J. Number Theory} {\bf 13} (2017), 
2373--2379.

\end{thebibliography}
\end{document}